\documentclass[a4paper,12pt,twoside]{amsart}
\usepackage{amsmath,amsthm}
\usepackage{amsfonts}
\usepackage{amssymb}
\usepackage[english]{babel}

\newtheorem{theo}{Theorem}[section]
\newtheorem{lem}[theo]{Lemma}
\newtheorem{prop}[theo]{Proposition}
\newtheorem{cor}[theo]{Corollary}

\renewcommand{\P}{{\mathbb P}}

\newcommand{\D}{{\mathcal D}}

\newcommand{\X}{{\mathcal X}}

\newcommand{\E}{{\mathbb E}}
\newcommand{\Y}{{\mathcal Y}}
\newcommand{\e}{{\text{e}}}
\newcommand{\C}{{\mathbb C}}

\newcommand{\NN}{{\mathbb N}}
\newcommand{\N}{{\mathbb N}}
\newcommand{\DD}{{\bar{\mathbb D}}}
\newcommand{\R}{{\mathbb R}}

\begin{document}

\title[Global CLT for Markov chains with normal transitions]
{Global central limit theorems for Markov chains\\with normal transition operators}

\author{Christophe Cuny}
\address{Univ Brest, CNRS, UMR 6205, Laboratoire de Math\'ematiques de Bretagne Atlantique, France}
\email{christophe.cuny@univ-brest.fr}

\author{Michael Lin}
\address{Department of Mathematics, Ben-Gurion University, Beer-Sheva, Israel}
\email{lin@math.bgu.ac.il}

\dedicatory{Dedicated to Magda Peligrad}

\subjclass[2010]{Primary: 60J05, 60F05. Secondary: 60J35, 47A35}
\keywords{Markov operators,  central limit theorems, uniform ergodicity,
 reversible Markov chains, normal operator,  annealed CLT, $L_2$-normalized CLT}

\begin{abstract}
Let $P$ be a Markov operator on a general state space $(S,\Sigma)$ with invariant probability
$m$, assumed ergodic. When $P$ is normal in $L_2(m)$, we prove the necessity of uniform
ergodicity, under some spectral conditions (satisfied when $P$ is symmetric), for the
annealed CLT to hold for {\it every} centered $ f \in L_2(m)$. We also obtain results without assuming normality under some Ritt-type conditions.
\end{abstract}

\maketitle

\section{Introduction}
Let $(S,\Sigma)$ be a measurable space and $P(x,A):\,(S,\Sigma) \to [0,1]$ a transition 
probability.  We denote also by $P$ the Markov operator $Pf(x) :=\int f(y) P(x,dy)$, defined 
for bounded measurable functions. We assume in this paper that $P$ has an invariant probability 
$m$, i.e.  $\int P(x,A)dm(x)=m(A)$ for any $A \in \Sigma$.
We will assume throughout the paper that $P$ is {\it ergodic} in $L_\infty(m)$, i.e.
$Pf=f\in L_\infty(m)$ implies that $f$ is a constant a.e.\footnote{This is compatible with 
the definition in ergodic theory, and weaker than the definition in \cite{MT}.}
We put $Ef:=\int f\,dm$ for $f\in L_1(m)$, and call $f$ {\it centered} when $Ef=0$.
\smallskip

Let $\Omega= S^\mathbb N$ be the space of trajectories, with $\sigma$-algebra
$\mathcal A:= \Sigma^{\otimes \mathbb N}$, and let $\P_m$ be the probability of the chain
with transition probability $P$ and initial distribution $m$.
By invariance of $m$,  the probability $\P_m$ is invariant under the shift on
$\Omega$, defined by  $\theta\{x_n\} =\{x_{n+1}\}$.
We denote by $X_k$ the $k$th coordinate  projection of $\Omega$ onto $S$, so
$X_k= X_0 \circ \theta^k$.
It is well-known that $\theta$ is ergodic on $(\Omega, \mathbb P_m)$ if and only if $P$ is ergodic
in $L_\infty(S,m)$ (e.g.  \cite[Theorem 2.5]{ALW}).
\smallskip


{\bf Definition.} A {\it real} centered $f \in L_2(m)$  satisfies the {\it annealed CLT}
 if in $(\Omega,\P_m)$ we have
\begin{equation} \label{clt}
\frac1{\sqrt n} \sum_{k=1}^n f(X_k) \overset{\mathcal D}\to{\mathcal N(0,\sigma^2)}.
\end{equation}
The definition includes the degenerate case $\mathcal N(0,0) \sim \delta_0$ (the Dirac
measure at $0$).
\smallskip

{\bf Definition.} A real centered $0 \ne f \in L_2(m)$ satisfies the {\it $L_2$-normalized CLT}
if, in $(\Omega, \mathbb P_m)$,
\begin{equation} \label{self-clt}
\frac1{\sigma_n(f)} \sum_{k=1}^n f(X_k) \overset{\mathcal D}\to{\mathcal N(0,1)},
\end{equation}
where $\sigma_n(f) =\|\sum_{k=1}^n f(X_k)\|_2$. The definition assumes $\sigma_n(f) >0$
for $n\ge N$.
\medskip

We denote by $L_{p,0}=L_{p,0}(m)=\{f \in L_p(m): \int_S f\,dm=0\}$ the space of all centered
real $f \in L_p(m)$. By the mean ergodic theorem, $L_{p,0}(m)=\overline{(I-P)L_{p}(m)}$
fot $1 \le p < \infty$.
\medskip

Gordin and Lifshits \cite{GL} proved that {\it if $P$ is ergodic and $f \in (I-P)L_2(m)$, then
$f$ satisfies the annealed CLT, with $\sigma^2=\|g\|_2^2-\|Pg\|_2^2$, when $f=(I-P)g$ with $g$
 centered}.

Consequently, if $L_{2,0}(m)=(I-P)L_{2,0}(m)$ (equivalently, if $I-P$ is invertible on 
$L_{2,0}(m)$) then the chain satisfies a {\bf global} annealed CLT:  
every $f \in L_{2,0}(m)$ satisfies the annealed CLT.

\medskip

The condition $L_{2,0}=(I-P)L_{2,0}$ is known to be equivalent to $P$ being {\it uniformly 
ergodic} in $L_2$, i.e.  $\|\frac1n\sum_{k=1}^n P^k-E\|_2 \to 0$ \footnote{Put $T:=P|L_{2.0}$; 
then, for large $n$, $(I-T)(\frac1n\sum_{k=1}^n\sum_{j=0}^{k-1} T^j)= 
I-\frac1n \sum_{k=1}^n  T^k$ is invertible, so $I-T$ is invertible too.}
\medskip

The above global CLT was proved by Rosenblatt \cite[p. 217]{Ro} when $\|P^n-E\|_2 \to 0$, 
which is equivalent to (exponetial) $rho$-mixing of the chain \cite[p. 207]{Ro}.
On the complex $L_2(m)$ the condition $\|P^n-E\|_2 \to 0$ is equivalent to $P$ having a
{\it spectral gap} -- the spectral radius of $P|L_{2,0}$ is strictly less than 1.
By \cite[Proposition 3.3]{DL4}, $P$ has a spectral gap if and only if it is uniformly ergodic
and has no unimodular eigenvalues $\ne 1$; see also Theorem \ref{totally-ergodic} in the Appendix.
\medskip

 In \cite[Section IV.8]{BI}, Gordin and Lifshits proved the quenched CLT (convergence of
 \eqref{clt} in $(\Omega,\P_x)$  for a.e. $x$) when  $f \in (I-P)L_2$, and even the functional
form of the quenched CLT. Hence the global functional quenched CLT holds when $P$ is
uniformly ergodic.
\medskip

The Law of the Iterated Logarithm  (LIL), and its functional form for $f\in (I-P)L_2(m)$ follows, for instance, from
\cite[Corollary 1]{ZW} (for the functional form, use also Strassen \cite{Strassen}). 


It follows that when $P$ is uniformly ergodic, every $0\ne f\in L_{2,0}(m)$ satisfies the LIL,
and in the symmetric case even the functional LIL.
\medskip

In this article we prove (Corollary \ref{symmetric1}) that {\it if $P$ is symmetric and
the global annealed CLT holds\footnote{Actually, it is sufficient that for all real centered $f\in L_2(m)$, $\big(\frac1n\sum_{k=1}^n f(X_k)\big)_{n\in \N}$ be tight}, then $P$ is uniformly ergodic}; hence the functional
quenched CLT and the functional LIL hold  "globally" as well. Moreover, {\it if $P$ is
symmetric and every non-zero $f \in L_{2,0}(m)$ satisfies a {\rm non-degenerate}
annealed CLT, then $\|P^n-E\|_2 \to 0$} (Corollary \ref{symmetric3}).
\bigskip

When $P$ is symmetric and Harris recurrent, the result of Corollary 3.2 may be derived from 
the recent preprint \cite{BRZ}, combining their Theorem 3.1 with Theorem 2.2 of \cite{DL4}. 
In our paper, Harris recurrence is not assumed.

\section{ Global annealed CLT and uniform tightness}



Put  $L_{p,0}=L_{p,0}(m):={\{f\in L_p(m)\,:\, \int_S f\, dm=0\}}$.
Throughout this section, $\X \subset L_{2,0}$ is a Banach space of functions or equivalence
classes (mod $m$) of functions on $(S,\Sigma,m)$, such that $\X$ is continuously embedded
in $L_{2,0}$ (i.e. $\|f\|_2 \le C\|f\|_\X$ for $f \in \X$), e.g. $\X=L_{p,0}$ for $p\ge 2$.

\medskip

For every $f\in \X$ and every $n\in \NN$, set $S_n(f):= f(X_0)+\ldots +f(X_{n-1})$.

\medskip

We first recall two results from \cite{CL23} that we state in a slightly different form.

\begin{prop}\label{prop-tightness}
Fix $\X$, and consider the following set of conditions in $(\Omega,\mathcal A,\P_m)$:
\begin{itemize}
\item [$(i)$] For every $f\in \X$, there exists a martingale $(M_n)_{n\in \NN}$ with stationary
and ergodic increments in $L_2(\P_m)$, such that for every $\varepsilon>0$,
	$\P_m(|S_n(f)-M_n|> \varepsilon\sqrt n)\underset{n\to +\infty}\longrightarrow 0$;
\item [$(ii)$] For every $f\in \X$, $(S_n(f)/\sqrt n)_{n\in \NN}$
converges in distribution to some normal law.
\item [$(iii)$] For every $f\in \X$, $(S_n(f)/\sqrt n)_{n\in \NN}$ is tight.
\item [$(iv)$] There exists a non-increasing function $c:\, (0,+\infty)\to [0,+\infty)$ with
$\lim_{a\to \infty}c(a) =0$, such that for every $f\in \X$
$$
\sup_{n\in \NN}\P_m \big(|S_n(f)|> a\|f\|_\X \sqrt n\big) \le c(a)\, .
$$
\end{itemize}
Then, $(i)\Rightarrow (ii) \Rightarrow (iii)\Rightarrow (iv)$.
\end{prop}



\begin{proof} $(i)\Rightarrow (ii)$ follows from the CLT for martingales in $L^2$ with stationary and ergodic increments. 
The proof of $(ii)\Rightarrow (iii) \Rightarrow (iv)$ is done in \cite[Proposition 7.1]{CL23}.
\end{proof}

\smallskip

\bigskip

\begin{prop}\label{prop-variance}
Assume that item $(iv)$ of Proposition \ref{prop-tightness} holds. Then, the set
$$
\Y:=\{f\in \X\, :\, \exists\,  \sigma_f\ge 0\, \mbox{such that}\, \, 
(S_n(f)/\sqrt{n})_{n\in \NN} \,\Rightarrow \mbox{ ${\mathcal N}(0,\sigma_f^2)$} \}
$$
is closed in $\X$. Moreover, there exists $C>0$ such that for every $f\in \Y$, $\sigma_f\le C\|f\|_\chi$. In particular, if $P\X\subset \X$, for every $f\in \X$, $\|f\|_2^2-\|Pf\|_2^2
\le C^2\|(I-P)f\|_\X^2$.
\end{prop}
\begin{proof}

 The proof that $\Y$ is closed is done in \cite{CL23}, see the middle of page 190. 
 
\medskip

Assume ad absurdum that there does not exist $C>0$ such that $\sigma_f\le C\|f\|_\X$ for
 every $f\in \Y$. Then, it is possible to find $(f_r)_{r\in \NN}$ in $\Y$ such that
 $\sigma_{f_r}=1$ for every $r\in \NN$ and $\|f_r\|_\X \to 0$.

\medskip

Using $(17)$ from \cite{CL23} with $K=1/\sqrt  {\|f_r\|_\X}$, we see that for every $t\in \R$,
$$
|{\rm e}^{it^2\sigma_{f_r}^2/2}-1|=\limsup_{n\to +\infty}|\E({\rm e}^{it S_n(f_r)/\sqrt n})-1|\le  
|t|\|f_r\|_\X^{1/2} +c(\|f_r\|_\X^{-1/2})\underset{r\to +\infty}\longrightarrow 0\, ,
$$ 
with $c$ the function in item $(iv)$, a contradiction.

 \medskip

Finally, it remains to prove that $\|f\|_2^2-\|Pf\|_2^2 \le C^2\|f-Pf\|_\X^2$, for every $f\in \X$.

\medskip

Let $f\in \X$. Set $g:= (I-P)f$ and, following Gordin and Lifshits \cite{GL}, set  $d_n:= f(X_n)-Pf(X_{n-1})$. Then $(d_n)_{n\in \NN}$
is a stationary sequence of martingale differences, hence $((d_1+\ldots +d_n)/\sqrt n)_{n\in \NN}$
converges in law to some normal law with variance $\E(d_1^2)=\|f\|^2-\|Pf\|^2$. Moreover,
$S_n(g)=(d_1+\ldots +d_n) + (Pf(X_0)-Pf(X_{n})$. Hence, $(S_n(g)/\sqrt n)_{n\in \NN}$ converges
in law to a normal law and $g\in \Y$.  In particular,
$\|f\|^2-\|Pf\|^2=\sigma_g^2\le C^2\|g\|_\X^2=C^2\|(I-P)f\|_\X^2$.
\end{proof}

\medskip

\begin{cor}\label{corollary}
Assume that for every $f\in \X$, $(S_n(f)/\sqrt n)_{n\in \NN}$ converges to some normal law 
and that $P\X\subset \X$. Then there exists $C>0$ such that for every $f\in \X$,
\begin{equation} \label{inequality}
\|f\|_2^2-\|Pf\|_2^2\le C^2 \|f-Pf\|_\X^2\, .
\end{equation}
\end{cor}

 \noindent {\bf Proof.} By Proposition \ref{prop-tightness}, we know that item $(iv)$ of that proposition holds and we conclude by Proposition \ref{prop-variance}.
 \hfill $\square$

\begin{prop} \label{equivalence}
Assume that $P\X\subset \X$ and that $(I-P)\X$ is dense in $\X$. Then, all items of
Proposition \ref{prop-tightness} are equivalent, and there exists $C>0$ such that \eqref{inequality} holds.
\end{prop}

\noindent {\bf Remark.} In particular, the corollary applies when $\X=L_{2,0}(m)$.

\medskip

\noindent {\bf Proof.} We already know that the items of Proposition \ref{prop-tightness} satisfy $(i)\Rightarrow (ii) \Rightarrow (iii)  \Rightarrow (iv)$. It remains to prove that, under our assumption, $(iv)\Rightarrow (i)$.

\medskip

Hence, let us assume that $(iv)$ holds. By Proposition \ref{prop-variance}, \eqref{inequality}
holds. Let $f\in \X$. By assumption, there exists $(f_r)_{r\in \NN}\in \X^\NN$ such that
$(I-P)f_r\underset{r\to +\infty}\longrightarrow f$. For every $r,n\in \NN$, set
$d_{r,n}:= f_r(X_n)-Pf_r(X_{n-1})$. Again, $(d_{r,n})_{n\in \N}$ is a stationary sequence of martingale differences and, for every $r,s, n\in \NN$,
$$
\|d_{r,n} -d_{s,n}\|_{2,\P_m}^2 =\|f_r-f_s\|_2^2-
\|P(f_r-f_s)\|_2^2\le C^2\|(I-P)f_r-(I-P)f_s\|_2^2\underset{r,s\to +\infty}\longrightarrow 0\, .
$$
Thus, for every $n\in \NN$, $(d_{r,n})_{r\in \NN}$ is Cauchy in $L_2(\P_m)$, hence converges, say to $d_n$. Then, $(d_n)_{n\in \NN}$ is a sequence of stationary and ergodic martingale differences in $L^2(\P_m)$. Setting $M_n:= \sum_{k=1}^n d_k$, it remains to prove that item $(i)$ is satisfied with this choice.

\medskip

Let $\varepsilon>0$. Let $n,r\in \NN$. Set $M_{r,n}:=
\sum_{k=1}^n d_{r,k}$. We have
\begin{gather*}
\P_m(|S_n(f)-M_n|>\varepsilon \sqrt n)\le \P_m(|S_n(f)-M_{r,n}|>\varepsilon \sqrt n/2)+\P_m(|M_{r,n}-M_n|>\varepsilon \sqrt n/2)\, .
\end{gather*}
Now, by Tchebychev's inequality,
$$
\sup_{n\in \NN}\P_m(|M_{r,n}-M_n|>\varepsilon \sqrt n/2)\le
\frac4{\varepsilon^2} \|d_{r,1}-d_1\|_{2.\P_m}^2\underset{r\to +\infty}\longrightarrow 0\, .
$$
Now, setting $g_r=f-(I-P)f_r$, we have $S_n(f)-M_{r,n}= Pf_r(X_0)-Pf_r(X_n)+S_n(g_r)$.

\medskip

Hence, using item $(iv)$ of Proposition \ref{prop-tightness}, we see that
$$
\P_m(|S_n(f)-M_{r,n}|>\varepsilon \sqrt n/2)   \le
$$
$$
P_m(|Pf_r(X_0)-Pf_r(X_n)|>\varepsilon \sqrt n/4)+
\P_m(|S_n(g_r)|>\varepsilon \sqrt n /4)
    \le \frac{64\|Pf_r\|_2^2}{\varepsilon^2n}+ c(\frac{\varepsilon}{4\|g_r\|_\X})\, .
$$
Then, we see that
$$
\limsup_{n\to +\infty} \P_m(|S_n(f)-M_n|>\varepsilon \sqrt n)\le
\sup_{n\in \NN}\P_m(|M_{r,n}-M_n|>\varepsilon \sqrt n/2)+c\big(\frac{\varepsilon}{4\|g_r\|_\X} \big) \underset{r\to +\infty}\longrightarrow 0\, .
$$
\hfill $\square$

\medskip

\medskip

\section{Application to the global CLT in $L_{2,0}(m)$}

In all of this section we assume that \emph{for  every $f\in L_{2,0}(m)$, $S_n(f)/\sqrt{n}$
{\rm does} converge in distribution to some normal law}, i.e. the global annealed CLT
holds in $L_{2,0}$. In view of Proposition \ref{equivalence} we may equivalently assume that for \emph{for every $f\in L_{2,0}(m)$, $(S_n(f)/\sqrt{n})_{n\in \N}$ is tight}. 

\medskip

  In particular, since we clearly have $PL_{2,0}(m)\subset L_{2,0}(m)$, and
$(I-P)L_{2,0}(m)$ is dense in $ L_{2,0}(m)$ by the mean ergodic theorem, Corollary
\ref{corollary} and Proposition \ref{equivalence} apply with $\X=L_{2,0}(m)$.

\medskip

Hence, assuming the global CLT in $L_{2,0}$ is equivalent to assuming that, for every $f\in L_{2,0}$, $(S_n(f)/\sqrt n)_{n\ge 1}$ is tight.

\subsection{The case of normal Markov operators}

We assume here, in addition to the global CLT in $L_{2,0}(m)$, that $P$ is normal, 
i.e. $P^*P=PP^*$.

Examples of normal Markov  operators with invariant probability are given by convolutions
defined by probability measures on the unit circle $\mathbb T$;
their uniform ergodicity is studied in \cite{DL4}.
\medskip

Fix $f\in L_{2,0}(m)$. Denote by $\rho_f$ the spectral measure associated with $f$
(on the Borel sets of $\DD:=\{z\in \C\, :\, |z|\le 1\}$) by the spectral theorem for normal
operators. Recall that $\rho_f$ is supported on $\sigma(P)$, the spectrum of $P$ (which 
is a subset of $\DD$), and that by the spectral theorem, for every $n\in \N$ and every 
complex numbers $a_0, \ldots a_n$,\footnote{See \cite[Corollary 8, page 898]{DS}, where 
$\rho_f(dz)=({\bf E}(dz)f,f)$ with the notations in item (ii) there.}
\begin{equation}\label{spectral}
\|a_0f+\ldots + a_n P^n f\|_2^2=\int_\DD |a_0+\ldots + a_nz^n|^2d\rho_f(z)\, .
\end{equation}
Since $P$ is ergodic and $f\in L_{2,0}(m)$, $\|f+\ldots +P^{n-1}f\|_2/n\underset{n\to +\infty}\longrightarrow 0$. In particular, by \eqref{spectral}, $\rho_f(\{1\})=0$.

\medskip

Applying \eqref{inequality} to $f+\ldots +P^rf$, the spectral theorem yields

$$
\int_\DD(1-|z|^2 )|1+\ldots +z^r|^2 d\rho_f(z) \le
C^2\int_\DD |1-z|^2|1+\ldots +z^r|^2 d\rho_f(z)\le 4C^2\|f\|_2^2\, .
$$

Using  Fatou's lemma and the fact that
$(1-|z|^2 )|1+\ldots +z^r|^2\underset{r\to +\infty} \longrightarrow \frac{1-|z|^2}{|1-z|^2}$
for $z \in \DD\backslash \{ 1\}$, we infer from $\rho_f(\{1\})=0$ that
\begin{equation}\label{necessity}
\int_{ \DD} \frac{1-|z|^2}{|1-z|^2} d\rho_f(z)\le 4C^2 \|f\|_2^2\, .
\end{equation}

\medskip

For every $K>0$ and $\beta\in (1/2,1]$, set
$$
\Lambda_{K,\beta}:=\{z\in \DD\, :\, |1-z|\le K(1-|z|)^\beta\}\, ,
$$
For every $\eta \in (0, 1]$, set
$$
\Delta_\eta:= \{z\in \DD\, :\, |1-z|\ge \eta\}\, .
$$

\medskip

Assume that the spectrum $\sigma(P)$ is included in $\Lambda_{K,\beta}\cup \Delta_\eta$ for
some $K>0$, some $\beta\in (1/2,1]$ and some $\eta\in (0,1]$. Then, since $\rho_f(\{1\})=0$,
\begin{gather*}
\int_\DD \frac{1+|z|}{|1-z|^{2-1/\beta}} d\rho_f(z) =
\int_{\Lambda_K\cup \Delta_\eta} \frac{1+|z|}{|1-z|^{2-1/\beta}} d\rho_f(z)\\ \le \frac{2}{\eta^{2-1/\beta}} \|f\|_2^2 + K^{1/\beta} \int_\DD \frac{1-|z|^2}{|1-z|^2} d\rho_f(z)\le L \|f\|_2^2\, ,
\end{gather*}
for some $L>0$.

\medskip

Hence, $\int_\DD \frac1{|1-z|^{2-1/\beta}}d\rho_f(z) <\infty$ which, 
by Theorem 4.4 of \cite{DL}, implies that
$f\in (I-P)^{1-1/2\beta} L_{2,0}(m) \subset L_{2,0}(m)$.

\medskip

We conclude that $L_{2,0}(m)=(I-P)^{1-1/2\beta} L_{2,0}(m)$, which, by items $(iii)$
and $(ii)$ of Proposition 2.1 of \cite{DL}, implies that $(I-P)L_{2,0}(m)= L_{2,0}(m)$;
hence $P$ is uniformly ergodic.

\medskip

We have proved the following.

\begin{theo} \label{normal1}
Let $P$ be an ergodic Markov operator which is normal in $L_2(m)$.
Assume that the spectrum of $P$ is included in $\Lambda_{K,\beta}\cup \Delta_\eta$ for
some $K>0$, some $\beta\in (1/2,1]$ and some $\eta\in (0,1]$. Then, the following are equivalent
\begin{itemize}
\item [$(i)$] For every $f\in L_{2,0}(m)$, $(S_n(f)/\sqrt{n})_{n\in \N}$ is tight; 
\item [$(ii)$] For every $f\in L_{2,0}(m)$, $(S_n(f)/\sqrt{n})_{n\in \N}$ converges to some normal law; 
\item [$(iii)$] There exists $C>0$ such that  $\|f\|^2-\|Pf\|^2\le C^2 \|(I-P)f\|^2$, for every $f\in L_{2,0}(m)$; 
\item [$(iv)$] There exists $C>0$ such that $\displaystyle \int_{ \DD} \frac{1-|z|^2}{|1-z|^2} d\rho_f(z)\le C^2\|f\|_2^2$, for every $f\in L_{2,0}(m)$;
\item [$(v)$] $P$ is uniformly ergodic.
\end{itemize}
\end{theo}

\smallskip

{\bf Remark.} If $P$ is normal and 1 is isolated in $\sigma(P)$ (i.e.
$\sigma(P) \setminus \{1\} \subset \Delta_\eta$ for some $\eta>0$), then by the spectral
theorem $P$ is uniformly ergodic\footnote{A contraction $T\ne I$
in Hilbert space with $\sigma(T)=\{1\}$ is never uniformly ergodic \cite[Theorem 6]{Ze}, e.g. $T:=(I+V)^{-1}$ on $L_2[0,1]$, with $V$ the Volterra operator \cite[p. 370]{Ze}.},
and the global CLT holds by \cite{GL}.

%
%
%

\begin{cor} \label{symmetric1}
Let $P$ be an ergodic Markov operator, which is symmetric in $L_2(m)$. Then
$(S_n(f)/\sqrt n)_{n\in \NN}$ converges in law to some normal law for every $f\in L_{2,0}(m)$
if and only if $P$ is uniformly ergodic.
\end{cor}
\begin{proof} If $P$ is uniformly ergodic, the global annealed CLT holds by \cite{GL}.
\smallskip

The converse follows from Theorem \ref{normal1}, since
 $\sigma(P) \subset [-1,0)\cup [0,1] \subset \Delta_1 \cup \Lambda_{1,1} $.
\end{proof}

{\bf Remark.} When $P$ is Harris recurrent and symmetric, Roberts and Rosenthal
\cite[Theorem 7]{RR1} proved that the global annealed "usual" CLT (the CLT with a given
formula for the variance of the limit) is equivalent to variance bounding (their
Theorem 6 shows that the Harris recurrence is implicit in their work);
by \cite{DL4} variance bounding is equivalent to uniform ergodicity. 
See also \cite{BRZ} for an improvement of Theorem 7 of \cite{RR1}. 

\begin{cor} \label{symmetric3}
Let $P$ be an ergodic Markov operator, which is symmetric in $L_2(m)$. Then
$(S_n(f)/\sqrt n)_{n\in \NN}$ converges in law to a {\rm non-degenerate} normal law for
every $0 \ne f\in L_{2,0}(m)$ if and only if $\|P^n-E\|_2 \to 0$.
\end{cor}
\begin{proof} When $\|P^n-E\|_2 \to 0$, the global non-degenerate annealed CLT holds by
\cite[Proposition 2.9]{CL23}.
\smallskip

For the converse, $P$ is uniformly ergodic by Corollary \ref{symmetric1}.  Since the CLT is
non-degenerate for every $0\ne f \in L_2(m)$, by \cite[Corollary 2.6]{CL23} $P^*P$ is ergodic,
so $P$ on the complex $L_2(m)$ has no unimodular eigenvalues except 1. 
By \cite[Proposition 3.3]{DL4}, $\|P^n-E\|_2 \to 0$.
\end{proof}

{\bf Remark.} Corollaries \ref{symmetric1} and \ref{symmetric3} provide positive answers
to problems 1 and 2 of \cite{CL23}, respectively, when the chain is reversible ($P$
is symmetric).
\medskip

In the case of normal operators, we were not able to prove the analogue of Corollary 
\ref{symmetric3}, without further restrictions on the spectrum. However, when $P$ is 
normal and satisfies the global CLT we can still say something more. 

\begin{prop}
Let $P$ be an ergodic and normal Markov operator. Assume that the coresponding chain 
satisfies the global CLT.  Then, for every $f\in L_{2,0}$, and $n\in \N$, 
$d_n:= \sum_{k\ge 0} (f(X_{n+k})-Pf(X_{n+k-1}))$ converges in $L^2(\Omega,\P_m)$, 
$(d_n)_{n\in \N}$ is a stationary and ergodic sequence of martingale differences such that, 
writing $M_nf:= d_1+\ldots +d_n$, for every $\varepsilon>0$,
\begin{equation}\label{approx}
\P_m(|S_n(f)-M_n f|\ge \varepsilon \sqrt n)\underset {n\to +\infty}\longrightarrow 0\, .
\end{equation}
Moreover, 
\begin{equation}\label{variance}
\E(d_1^2)=\int_\DD \frac{1-|z|^2}{|1-z|^2}d\rho_f(z)\,.
\end{equation}
 In particular, we may have a degenerate CLT if and only if $P^*P$ is not ergodic. 
\end{prop}

{\bf Remarks.} 1. The fact that we may have a degenerate CLT when $P^*P$ is not ergodic follows 
from  \cite[Corollary 2.8]{CL23} (which does not need uniform ergodicity for that implication).

2. The martingale difference  $d_1$ appearing in the proposition is exactly the Gordin 
martingale as described in (1.3) of \cite{DM}. For many situations for which that martingale 
difference is the "right one" to be considered to prove a CLT, see \cite{DM}.
\smallskip

\begin{proof} We first prove the convergence in $L_2(\P_m)$. Let $q>p$ be natural numbers. 
We have
\begin{align*}
\|\sum_{k=p}^q f(X_{n+k})-Pf(X_{n+k-1})\|_2^2  & = \int_\DD \big|\sum_{k=p}^q z^k\big|^2(1-|z|^2)d\rho_f (z)\\
  &  \le \int_\DD\frac{|1-z^{q+1-p}|^2(1-|z|^2)}{|1-z^2|}d\rho_f(z)\, ,
\end{align*}
hence, by \eqref{necessity} and the  Lebesgue dominated theorem, we have a Cauchy sequence converging in $L_2(\P_m)$, say to $d_n$ such that 
$\E(d_1^2)=\int_\D\frac{1-|z|^2}{|1-z|^2}d\rho_f(z)$. 
In particular, $\E(d_1^2)\neq 0$ if $P^*P$ is ergodic by the 
Lemma \ref{ergodic-normal} below.

\smallskip

Now, we prove \eqref{approx}. Let $ f_r:=\sum_{k=0}^r P^k f$, for $r\in \N_0$, and $\tilde f_r :=\frac1r\sum_{k=0}^{r-1}f_k$, for $r\in \N$.  Let us notice that we are not exactly in the situation of the proof of Proposition \ref{equivalence}, since we are not sure that $((I-P)f_r)_{r\in \N_0}$ converges to $f$. Now, using the fact that $P$ is ergodic and the identity
$$
I = (I-P)\frac{1}{r}\sum_{k=1}^r\sum_{\ell =0}^{k-1}P^kf
+\frac1r\sum_{k=0}^{r-1}P^k\, ,$$
we see that the Ces\`aro averages of $((I-P)f_r)_{r\in \N}$ congerges to $f$, and then the 
proof is similar to the one  done for Proposition \ref{equivalence}.  
\end{proof}

{\bf Remarks.} 1. In view of the proposition and of Theorem \ref{normal1}, one may
 wonder whether, for normal operators without the spectral condition, item $(iv)$ of
Theorem \ref{normal1} implies that $P$ is uniformly ergodic. 
This is not the case as one may see from the following example. Assume that $Pf=f\circ T$
 for an invertible ergodic measure preserving transformation $T$ of an atomless probability
 space. Then $P$ is unitary, and its spectrum is the full unit circle \cite{IT}, \cite{Foi}; 
item $(iv)$ of Theorem \ref{normal1} trivially holds, while $P$ is not uniformly ergodic 
since 1 is not isolated in the spectrum.

2. Since $1-|z|^2 \le 2|1-z|$ for $|z| \le 1$, the finiteness of \eqref{variance} can be 
strengthened to 
\begin{equation} \label{square-root}
\int_{\bar{D}} \frac1{|1-z|} d\rho_f(z) < \infty.
\end{equation}
In \cite[Corollary IV.7.2]{BI}, Gordin and Lifshits proved that 
if $f$  satisfies \eqref{square-root}, then $(S_n(f)/\sqrt{n})_{n\in \mathbb N}$ converges 
in distribution to a normal law.
For $P$ normal, \eqref{square-root} is equivalent  to $f \in \sqrt{I-P}L_2(m)$, and
$\sqrt{I-P}L_2(m)$ is closed if and only if $(I-P)L_2(m)$ is closed
\cite[Proposition 2.1(iii)]{DL}, which is equivalent to uniform ergodicity. 
We then have for $P$ normal, without any condition on $\sigma(P)$, that {\it $P$ is 
uniformly ergodic if and only if every $f \in L_{2,0}(m)$ satisfies \eqref{square-root}.}

\begin{lem}\label{ergodic-normal}
Let $P$ be an ergodic and normal Markov operator. The following are equivalent 
\begin{itemize}
\item [$(i)$] $P^*P$ is ergodic; 
\item [$(ii)$] for every $f\in L_{2,0}$, $\|Pf\|_2=\|f\|_2\Rightarrow f\equiv 0$;
\item [$(iii)$] for every $f\in L_{2,0}$, $\|P^nf\|_2\underset{n\to +\infty}\longrightarrow 0$;
\item [$(iv)$] for every $f\in L_{2,0}$, $\rho_f(\{z\in \C\, :\, |z|=1\})=0$. 
\end{itemize} 
\end{lem}
\noindent {\bf Proof.}
The equivalence of $(i)$ and $(ii)$ follows from Lemma 4.5 of \cite{CL23} (and does not need 
$P$ to be normal), and the equivalence of $(i)$ and $(iii)$ follows from Lemma 4.2 of \cite{CL23}. 
Then, the equivalence of $(iii)$ and $(iv)$ is a straightforward 
consequence of \eqref{spectral}. \hfill $\square$

\subsection{The case of Markov operators of Ritt-type}

We say that the operator $P$ is {\it quasi-Ritt} of order $\beta\in (1/2,1]$ if there exists
$K>0$ such that $n^\beta\|P^n-P^{n+1}\|_2 \le K$, for every $n\in \NN$.

We simply say that the operator is Ritt if it is quasi-Ritt of order 1.
\medskip

The above definition is motivated by several papers; e.g. \cite{CCL1},\cite{CoLi1} and
references there.

Recall that when $P$ is quasi-Ritt
\footnote{All unidexed norms are $L_2(m)$-norms.}, $\|P^nf\| \underset{n\to\infty}\longrightarrow 0$ for every 
$f\in L_{2,0}(m)$. Indeed, since  $\sup_{n\in \N}\|P^n\|\le 1<\infty$, the set $\{f\in L_{2,0}(m)\, :\, \|P^nf\|\underset{n\to \infty}\longrightarrow  0\}$ is closed but, by the quasi-Ritt property, it contains the dense set $(I-P)L_{2,0}(m)$.
\medskip

{\bf Remark.} Rosenblatt's condition $\|P^n-E\|_2 \to 0$ implies $\|P^n-E\|_2 \le C\rho^n$
for some $\rho<1$ (e.g. \cite[Proposition 3.1]{DL4}), so in particular $P$ is Ritt.
\smallskip

\begin{prop} \label{quasi-ritt}
Assume that $P$ is quasi-Ritt of order $\beta$ for some $\beta\in (1/2,1]$ and that for every
$f\in L_{2,0}(m)$, $(S_n(f)/\sqrt n)_{n\in \NN}$ converges in distribution to some normal law.
Then $\|P^n-E\|_2 \underset{n\to +\infty}\longrightarrow 0$.
\end{prop}
\begin{proof} By Proposition \ref{prop-variance}, there exists $C>0$ such that
$\sigma_f\le C\|f\|_2$ for every $f\in L_{2,0}(m)$.
Fix $f\in L_{2,0}(m)$ and set $g:= (I-P)f$.
By the Gordin-Lifshits CLT \cite{GL},  $(S_n(g)/\sqrt n)_{n\in \NN}$ converges in law
to a normal distribution with variance $\sigma_g^2=\|f\|^2-\|Pf\|^2$ and we infer that
\begin{equation}\label{inequality1}
\|f\|^2-\|Pf\|^2 =\sigma_g^2 \le C^2\|g\|^2= C^2\|f-Pf\|^2\, .
\end{equation}

Applying this to $P^nf$, and using the quasi-Ritt assumption, we see that
$$
\|P^nf\|^2- \|P^{n+1}f\|^2\le \frac{(CK)^2}{n^{2\beta}}\|f\|^2\, .
$$

Using that $\|P^nf\|\underset{n\to +\infty}\longrightarrow 0$ and that the series
$\sum_{n}\frac1{n^{2\beta}}$ converges, we see that there exists $L>0$ such that
$$
\|P^nf\|^2\le \frac{L}{n^{2\beta-1}}\|f\|^2\, .
$$
Hence, $\|P^n_{|L_{2,0}}\|^2\le \frac{L}{n^{2\beta-1}}$ and  the proof is finished.
\end{proof}

\medskip

 {\bf Remarks.} 1. Instead of the quasi-Ritt property, it is enough to assume that
there exists $\varphi$ such that $\sup_{n\in \NN}\varphi(n) \|P^n-P^{n+1}\|<\infty$  and
$\sum_{n\in \NN}\frac1{\varphi^2(n)}<\infty$. 

2. Note that normality of the operator $P$ is not assumed.

\medskip

{\bf Definition.} A centered $0 \ne f \in L_2(m)$ satisfies the {\it  Functional
$L_2$-normalized CLT} if, in $(\Omega,\mathbb P_m)$,
\begin{equation} \label{fclt}
\Big\{	\frac1{\sigma_n(f)} \sum_{k=1}^{[nt]} f(X_k), \ t\in[0,1]\Big\}
\overset{\mathcal D}\to{W(t) },
\end{equation}
where $W(t), \ t \ge 0$, is the standard Brownian motion.

\begin{theo} \label{symmetric}
Let $P$ with invariant probability $m$  be  symmetric on $L_2(m)$ and ergodic.
Assume that $P$ is positive definite. Then the following are equivalent:

(i) Every $f \in L_{2,0}(m)$ satisfies the annealed CLT.

(ii) Every non-zero $f \in L_{2,0}(m)$ satisfies a non-degenerate annealed CLT.

(iii) $P$ is uniformly ergodic on $L_2(m)$.

(iv) $\|P^n-E\|_2 \to 0$.

(v) Every $0 \ne f \in L_{2,0}(m)$ satisfies the {\em functional} $L_2$-normalized CLT.

\noindent
If either condition holds, then:

(vi) Every $0 \ne f \in L_{2,0}(m)$ satisfies the $L_2$-normalized CLT.

\noindent
When $P$ is also Harris recurrent, then (vi) is equivalent to all the other conditions.
\end{theo}
\begin{proof} Clearly (iv) implies (iii), and (iii) implies (i) by Gordin-Lifshits \cite{GL}.

(iv) implies (ii) by \cite[Proposition 2.9]{CL23}, since $P$ is symmetric. Obviously (ii)
implies (i).

For the sake of completeness we prove that $P$ satisfies $ \sup_n n\|P^n-P^{n+1}\|_2< \infty$
(shown implicitely by Stein \cite{St}). Since $P$ is positive definite,
$\sigma(P) \subset [0,1]$. Since
$$
\max_{t \in [0,1]} t^n(1-t)=\Big(\frac{n}{n+1}\Big)^n\Big(1 -\frac{n}{n+1}\Big)=
\Big(\frac{n}{n+1}\Big)^n \cdot \frac1{n+1} \le \frac1\e \cdot \frac1{n+1}\ ,
$$
the spectral theorem yields
$$
\|P^n(I-P) f\|^2=  \int_0^1 (t^n-t^{n+1})^2 \langle E(dt)f,f\rangle \le
\frac1{\e^2} \cdot \frac1{(n+1)^2} \|f\|^2;
$$ 
hence $n\|P^n-P^{n+1}\| \le \frac1{\e}$.

Now (i) implies (iv) by Proposition \ref{quasi-ritt}.

By \cite[Theorem 1]{PU}, (v) implies that any $0\ne f \in L_{2,0}(m)$ satisfies
	$\sigma_n^2(f)= nL(n)$ with some $L(n)$ slowly varying.
	Fix $\gamma \in (0,1)$ and put $\epsilon=(1-\gamma)/2$.
It was shown in the proof of \cite[Theorem 3.9]{CL23} that
$\|\frac1n\sum_{k=1}^n P^k f\| \le C_f/n^\epsilon$, and by the Banach-Steinhaus
theorem, $P$ is uniformly ergodic. Thus (v) implies (iii).

When (iv) holds, the global $L_2$-normalized CLT (vi) follows from
\cite[Proposition 2.9]{CL23}, and by \cite{GL} $\frac1n\sigma_n^2(f)$ converges for any
centered $f \in L_2$; for $f \ne 0$ the limit is not zero by \cite[Proposition 2.9]{CL23}.
By \cite{LPP}, these two properties yield \eqref{fclt}. See also \cite[Theorem 14.9]{MPU}.
Hence (iv) implies (v).
\smallskip

When $P$ is also Harris recurrent, (vi) implies (iv) by \cite[Proposition 4.11]{CL23}.
\end{proof}

{\bf Remarks.} 1. The equivalences of Theorem \ref{symmetric} hold if $P$ is symmetric
and $-1 \notin \sigma(P)$.

2. The equivalence of (i) and (iii), without any assumption on the spectrum, is in Corollary
\ref{symmetric1}.

3. The equivalence of (ii) and (iv), without any assumption on the spectrum, is in Corollary
\ref{symmetric3}.



\begin{prop} \label{dedecker}
There exists an aperiodic symmetric Harris operator $P$ for which the global annealed CLT fails,
but every $f \in L_{p,0}(m)$, $p>2$, satisfies the annealed CLT.
\end{prop}
\begin{proof}
Putting $\zeta_n=\e^{-\sqrt{n}}$ in Bradley's \cite[Theorem 3.3]{Br4}, we
obtain a reversible positive recurrrent chain on a countable state space with mixing
coefficients $\alpha(n) \le \beta(n) \le \e^{-\sqrt{n}}$. Its Markov operator $P$ is
symmetric aperiodic,
but by the construction in \cite{Br4}, some $f \in L_{2,0}(m)$ does not satisfy the
annealed CLT; hence, by \cite{GL}. $P$ is not uniformly ergodic. With the given rate
$\alpha(n) \le \e^{-\sqrt{n}}$, \cite[Proposition 3.5]{CL23} yields the annealed CLT
for every $f \in L_{p,0}(m)$, $p>2$.
\end{proof}

{\bf Remarks.} 1. The operator $P$ in Proposition \ref{dedecker} is not geometrically
ergodic, since in the symmetric case geometric ergodicity implies $\|P^n-E\|_2 \to 0$
by \cite[Theorem 2]{RT}.

2. By \cite[Proposition 3.2]{CL23}, geometric ergodicity implies that every centered real
$f$ with $\int |f|^2 \log^+|f|dm< \infty$ satisfies the annealed CLT. We do not know
if this condition implies geometric ergodicity when $P$ is symmetric Harris aperiodic.

3. When $P$ is normal in $L_2(m)$ and ergodic, $\|P^n-E\|_2 \to 0$ if (and only if)
the mixing coefficients $\alpha(n)$ tend to 0 (at least) exponentially fast, by
\cite[Theorem 3.3]{CL23}.  H\"aggstr\"om's example \cite{Hag} is geometrically ergodic
with $\lim_n \|P^n-E\|_2 >0$, so it is not normal in $L_2(m)$.

4. In the proof of Proposition \ref{dedecker}, we could have used Bradley's
\cite[Example 1]{Br3} with  the parameters $a=1/2$, $q=1$ and $b=c=0$.
\medskip

\begin{cor}
There exists an aperiodic symmetric positive definite Harris operator $P$ for which
the global annealed CLT fails, $P$ is not uniformly ergodic in $L_2(m)$, but every
$f \in L_{p,0}(m)$.  $p>2$, satisfies the annealed CLT.
\end{cor}
\begin{proof} Let $Q$ be the operator of the proof of Proposition \ref{dedecker} and
put $P=Q^2$.  By aperiodicity of $Q$, the positive definite $P$ is aperiodic.
Since $\lim_n\|Q^n-E\|_2 >0$, also $\lim_n\|P^n-E\|_2 > 0$, Then, by Theorem
\ref{symmetric}, $P$ does not have the global CLT, and by \cite{GL} $P$ is not
uniformly ergodic.  By \cite[Formula (13)]{CL23}, the mixing coefficients satisfy
$\alpha_P(n) =\alpha_Q(2n) \le \e^{-\sqrt{2n}}$, so \cite[Proposition 3.5]{CL23}
yields the annealed CLT for every $f \in L_{p,0}(m)$, $p>2$.
\end{proof}

%
	
\medskip

Cohen and Lin \cite[Theorem 2.4]{CoLi1} gave sufficient conditions for a contraction on a
complex Hilbert space to be quasi-Ritt. They define quasi-Stolz regions $D(a,b)$ in the closed
unit disk, which generalize Stolz regions\footnote{A Stolz region is the closed convex hull of the
point $1$ and a centered disk of radius $\rho<1$.} (see \cite{CoLi1}).  By \cite[Lemma 5.2]{CoLi1},
we have
$$
D(\frac1\beta,b) \subset \{z \in \bar{\mathbb D}: |1-z| \le 2^\beta C(1-|z|)^\beta\}=
\Lambda_{2^\beta C,\beta}.
$$
Let us show that we have a converse inclusion.

\begin{prop}
\begin{itemize}
\item [$(i)$] For every $\beta\in (1/2,1)$ and $K>0$,
$\Lambda_{K,\beta}\subset D(\frac1{\beta}, \frac1{2K^{1/\beta}})$ ;
\item [$(ii)$] For every $K>0$, $\Lambda_{K,1}\subset D(1, \frac{K+1}{2K})$.
\end{itemize}
\end{prop}
\begin{proof}  Let $z\in \Lambda_{K,\beta}$ for some $K>0$ and some $\beta\in (1/2,1]$.
Write $z=x+iy=1-u+iy$. Then, $|1-z|^2 =(1-x)^2+y^2\ge y^2$ and $1-|z|^2 = 2u-u^2-y^2\le 2u$.
By assumption on $z$, since $1-|z|\le 1-|z|^2$, we infer that $y^2 \le K^2(2u)^{2\beta}$.
Hence, $\frac{|y|^{1/\beta}}{ 2K^{1/\beta}}\le 1-x$, and item $(i)$ follows. Item $(ii)$
follows by similar computations using the bound $1-|z|^2  \le 2u-y^2$.
\end{proof}

\begin{prop} \label{normal}
Let $P$ be a  normal operator on a complex Hilbert space. Then:

(i) $P$ is quasi-Ritt of order $\beta \in(\frac12,1)$ if and only if
$\sigma(P) \subset D(\frac1\beta,b)$ for some $b>0$.

(ii) $P$ is Ritt (quasi-Ritt of order 1) if and only if $\sigma(P)$ is contained in a Stolz region.
\end{prop}
Part (i) is proved in \cite[Corollary 2.7]{CoLi1}.
Part (ii) is proved in \cite[Proposition 2.5]{CCL1}.
\medskip

\begin{prop} \label{normal-quasi}
Let $P$ with invariant probability $m$  be normal on $L_2(m)$ and ergodic.
Assume that $P$ is quasi-Ritt of order $\beta \in (\frac12,1]$. Then the following are equivalent:

(i) Every $f \in L_{2,0}(m)$ satisfies the annealed CLT.

(ii) Every non-zero $f \in L_{2,0}(m)$ satisfies a non-degenerate annealed CLT.

(iii) $P$ is uniformly ergodic on $L_2(m)$

(iv) $\|P^n-E\|_2 \to 0$.

\noindent
If either condition holds, then:

(v) Every $0 \ne f \in L_{2,0}(m)$ satisfies the $L_2$-normalized CLT.

\noindent
When $P$ is also Harris recurrent, then (v) is equivalent to the other conditions.
\end{prop}
\begin{proof}
(i) implies (iv) by  Proposition \ref{quasi-ritt}.
(iv) implies (ii) by \cite[Proposition 2.9]{CL23}. Obviously (ii) implies (i).
Also (iv) implies (iii), which implies (i) by \cite{GL}.

The $L_2$-normalized CLT follows from (iv) by \cite[Proposition 2.9]{CL23}.
\smallskip

When $P$ is also Harris recurrent, (v) implies (iv) by \cite[Proposition 4.11]{CL23}.
\end{proof}

{\bf Remarks.} 1.   Condition (ii) alone, without the  quasi-Ritt assumption, implies that
 $P^*P$ is ergodic; see the proof of \cite[Corollary 2.8]{CL23}.
For $P$ normal, $P^*P$ is ergodic if and only if $\|P^ng-g\|_2 \to 0$ for every 
$g \in L_2(m)$, by \cite[Lemma 4.2]{CL23}.

2. If we do not assume normality, conditions $(i)$, $(iii)$ and $(iv)$ remain equivalent,
but they are no longer equivalent to $(ii)$, as shown by the example after Remark 2.9 of 
\cite{CL23}. 

\begin{cor}
Let $P$ with invariant probability $m$  be normal on $L_2(m)$ and ergodic. Then
$\|P^n-E\|_2 \to 0$ if and only if $P$ is Ritt and every $f \in L_{2,0}(m)$ satisfies the
annealed CLT.
\end{cor}
\begin{proof} As noted above, $\|P^n-E\|_2 \to 0$ implies that $P$ is Ritt,  and the global
annealed CLT follows from \cite{Ro}. The converse is (i) implies (iv) in Proposition 
\ref{normal-quasi}. 

Note that $\|P^n-E\|_2 \to 0$ implies that every $0\ne f \in L_{2,0}(m)$ satisfies a
non-degenerate annealed CLT \cite[Proposition 2.9]{CL23}.
\end{proof}

The next proposition presents a global CLT when $P$ is Ritt, without assuming normality.

\begin{prop}
Let $P$ with invariant probability $m$  be  ergodic. If $P$ is Ritt in the complex $L_2(m)$ with
$\limsup_{n\to\infty} n\|P^{n+1}-P^n\|_2 < 1/\e$, then  every $f \in L_{2,0}(m)$ satisfies the
annealed CLT.
\end{prop}
\begin{proof} By Kalton et al. \cite[Theorem 3.1]{KMOT}, $P^n$ converges in operator norm.
Ergodicity of $P$ implies that the limiting projection is $E$, so $\|P^n-E\|_2 \to 0$.
The global annealed CLT then follows from Rosenblatt's CLT  \cite[p. 217]{Ro}.
\end{proof}

\medskip

\section{Appendix: Peripheral spectra of Markov operators}

\subsection{The Katznelson-Tzafriri theorem}

A power-bounded operator $T$ on a complex Banach space has spectral radius $r(T) \le 1$.
When $r(T)=1$, we calll $\sigma(T)\cap\mathbb T$ the {\it peripheral spectrum} of $T$;
if $r(T)<1$ then $\|T^n\| \to 0$.

Katznelson and Tzafriri \cite{KT} proved the following result.

\begin{theo} \label{kt}
Let $T$ be a power-bounded operator on a complex Banach space. Then
$\|T^{n}-T^{n+1}\| \to 0$ if and only if $\sigma(T)\cap\mathbb T \subset\{1\}$.
\end{theo}

{\bf Remarks.} 1. Obviously, if $T^n$ converges in operator norm, then
$\|T^n-T^{n+1}\| \to 0$.

2. The "only if" part of Theorem \ref{kt} follows from the spectral mapping theorem:
if $\lambda \in \sigma(T)\cap\mathbb T$, then
$$
|\lambda-1|=|\lambda^n(\lambda-1)| \le r(T^n-T^{n+1}) \le \|T^n-T^{n+1}\| \to 0.
$$

3. When $T$ is a normal contraction in Hilbert space, the "if" part of Theorem \ref{kt}
follows from the spectral theorem, since for $\lambda \in \sigma(T)$ the assumption
implies $|\lambda^n(1-\lambda)| \to 0$.

4. When $T$ is a power-bounded {\em positive} operator on a complex Banach lattice,
$\|T^n-T^{n+1}\| \to 0$ if and only if $\inf_n \|T^n-T^{n+1}\| < \sqrt3$, by
\cite[Theorem 5]{L}.

\begin{cor} \label{reflexive}
Let $T$ be a mean ergodic power-bounded operator on a complex Banach space $X$ (e.g.
$X$ is reflexive).  If $\sigma(T)\cap\mathbb T \subset\{1\}$, then $T^n$ converges in
the strong operator topology.
\end{cor}
\begin{proof} By the mean ergodic theorem, $X=\{y\in X:Ty=y\}\oplus\overline{(I-T)X}$.
Apply now Theorem \ref{kt}.
\end{proof}

\medskip

{\bf Definition.} A power-bounded $T$ is a {\it Ritt operator} if $\sup_n n\|T^n(I-T)\|<\infty$.

This means the rate $\|T^n-T^{n+1}\|=O(1/n)$ in Theorem \ref{kt}.
By \cite[Theorem 3.1]{KMOT}, either $\limsup_n n\|T^n-T^{n+1}\| \ge 1/\rm e$, or $T^n$
converges in operator norm (and then we have $\|T^n-T^{n+1}\| \le M\rho^n$ for some $\rho<1$).
In particular, $\|T^n-T^{n+1}\|=o(1/n)$ implies operator-norm convergence of $T^n$.

 Nagy and Zem\'anek \cite{NZ} proved that a power-bounded Ritt operator $T$ has its spectrum 
 contained in a closed Stolz region. A normal contraction in a complex Hilbert
 space  $H$ is Ritt if and only if its spectrum is contained in a Stolz region
 \cite[Proposition 2.5]{CCL1}. Hence a symmetric contraction which is positive definite is Ritt,
a fact already implicit in Stein \cite{St}. In fact, a symmetric contraction on $H$  is Ritt
if and only if $-1 \notin \sigma(T)$.

Other rates of convergence in Theorem \ref{kt} (slower than $1/n$) were studied by Seifert \cite{Seifert} and by  Cohen 
and Lin \cite{CoLi1}.  When $T$ is normal in $H$, such rates were characterized by $\sigma(T)$
being a subset of a certain quasi-Stolz region.

\medskip

Now let $P$ be a Markov operator with invariant probability $m$, which is assumed ergodic.
Then $P$ is also a contraction of the complex $L_2(m)$, with spectrum $\sigma(P)$.
Since $P{\mathbf 1}={\mathbf 1}$, $1 \in \sigma(P)$. Corollary \ref{reflexive} yields
that if $\sigma(P)\cap \mathbb T=\{1\}$, then every $f \in L_2(m)$ satisfies
$\|P^nf-Ef\|_2 \to 0$ and, since also $\sigma(P^*)\cap \mathbb T=\{1\}$,
$\|P^{*n}f-Ef\|_2 \to 0$. By approximation, $\lim_n\|P^nf-Ef\|_1 = 0=\lim_n\|P^{*n}f-Ef\|_1 = 0$
for every $f \in L_1(m)$. This last property is equivalent to triviality (mod $\mathbb P_m$)
of both tail $\sigma$-fields of the two-sided stationary chain defined by $P$ 
\cite[Th\'eor\`eme 6]{D}.
\smallskip

{\bf Example 1.} {\it A non-Ritt normal Markov operator with $\sigma(P)\cap\mathbb T=\{1\}$.}

\noindent
Let $\theta$ be an invertible measure preserving transformation of a standard probability space
$(S,\Sigma,m)$, with Koppman operator $Uf=f\circ\theta$. Then on the complex $L_2(m)$ $U$
satisfies $\sigma(U)=\mathbb T$ \cite{IT}, \cite{Foi}. Define  $P:=\frac12(I+U)$. Then $m$ is
invariant for $P$, and $P$ is normal and ergodic in $L_2(m)$. By the spectral mapping theorem,
$\lambda \in \sigma(P)$ if and only if $\lambda=\frac12(1+z)$ for some $z \in\mathbb T$,
 which shows that $\sigma(P)\cap\mathbb T=\{1\}$. For $0 \ne \lambda \in \sigma(P)$,
 $\lambda \ne 1$, we have $Re(\lambda) >0$ and $Im(\lambda) \ne 0$, so $\sigma(P)$ is not
 contained in a Stolz (or quasi-Stolz) region. In fact, $\sigma(P)$ is the circle of center
 $(\frac12,0)$ and radius  $\frac12$. In particular, $P$ is not Ritt.
 \smallskip

 {\bf Remark.}  In Example 1 the point $1$ is not isolated in $\sigma(P)$, so
 $\lim_n \|P^n-E\|_2 >0$.
 \medskip

 {\bf Example 2.} {\it A Ritt symmetric Markov operator with $\lim_n\|P^n-E\|_2 >0$.}

 \noindent
 Let $\theta$ and $U$ be as in Example 1, and define $P=\frac12I +\frac14 U+\frac14 U^{-1}$.
 Then $P$ is symmetric, and, by the spectral mapping theorem,  $\sigma(P)=[0,1]$. Hence
 $P$ is positive definite, so it is Ritt. Since $1$ is not isolated
 in $\sigma(P)$,
 $\lim_n \|P^n-E\|_2>0$.

\subsection{Peripheral spectra in $L_p$, $1<p< \infty$}

A Markov operator $P$ with invariant probability $m$ is a contraction of all $L_p$ spaces.
We denote by $\sigma(P_{|L_p})$ the spectrum of $P$ as a contraction on the {\em complex}
$L_p(m)$ space, $p\ge 1$.
\smallskip

The special case of Corollary \ref{reflexive}, when $P$ is an ergodic Markov operator
with invariant probability, had been proved in \cite{LS} (before \cite{KT}), assuming
$\sigma(P_{|L_1})\cap\mathbb T=\{1\}$.
\smallskip

Rosenblatt \cite[p. 211]{Ro} proved (for the ergodic case) that if $\|P^n-E\|_r \to 0$ for
some $1 \le r \le \infty$, then $\|P^n-E\|_p \to 0$ for every $1<p< \infty$.
The following is in the same spirit (without any ergodicity assumption).

\begin{theo} \label{gluck}
Let $P$ be a Markov operator with invariant probability $m$. Then for $1<p< \infty$
the peripheral spectra $\sigma(P_{|L_p})\cap\mathbb T$ are all the same.
\end{theo}
\begin{proof} We first prove the following.

\noindent{\bf Claim.} {\it Let $T$ be a contraction on a Banach space. Then $I-T$ is 
invertible if and only if $\|\frac1n\sum_{k=1}^{n} T^k\| \to 0$.} The "only if" is obvious. 
For the "if" part, for large $n$ fixed, $\|\frac1n\sum_{k=1}^{n} T^k\| <1$, so 
$$
(I-T)\Big(\frac1n\sum_{k=1}^n\sum_{j=0}^{k-1}T^j \Big)=I-\frac1n\sum_{k=0}^{n-1} T^k
$$
is invertible. Hence $I-T$ is invertible.
\smallskip

To prove the theorem, we show that given $1<p_0<p<\infty$  and $\lambda\in\mathbb T$,
$\lambda I-P$ is invertible in $L_{p_0}$ if and only if it is invertible in $L_p$. 
	
Assume first that $\lambda I-P$ is invertible in $L_{p_0}$,  so
$\|\frac1n \sum_{k=1}^n (\bar\lambda P)^k\|_{p_0} \to 0$.  Take $p_1 \in (p,\infty)$. 
Then $\frac1{p_1} < \frac1p < \frac1{p_0}$, so $\frac1p= t\frac1{p_1} + (1-t)\frac1{p_0}$ 
for some $0<t<1$. By the Riesz-Thorin convexity theorem (e.g. \cite[Theorem XII.1.11]{Z}),
$$
\Big\|\frac1n \sum_{k=1}^n (\bar\lambda P)^k\Big\|_p \le
	\Big\|\frac1n \sum_{k=1}^n (\bar\lambda P)^k\Big\|_{p_0}^{1-t} \cdot
	\Big\|\frac1n \sum_{k=1}^n (\bar\lambda P)^k\Big\|_{p_1}^{t}  \le
	\Big\|\frac1n \sum_{k=1}^n (\bar\lambda P)^k\Big\|_{p_0}^{1-t} \to 0.
$$
Hence $\lambda I-P$ is invertible in $L_p$.  we conclude that
$\sigma(P_{|L_p})\cap\mathbb T \subset \sigma(P_{|L_{p_0}})\cap\mathbb T$. 

For the opposite inclusion, i.e. when $\lambda I-P$ is invertible in $L_{p}$, 
we choose $p_2 \in(1,p_0)$ and interpolate for $p_2<p_0<p$.
\end{proof}

{\bf Remarks.} 1. Theorem \ref{gluck} was first proved by Jochen Gl\"uck (private 
communication), using (in addition to the interpolation) a deep result from \cite{KT}.

 2. We always have $\sigma(P_{|L_p})\cap\mathbb T \subset\sigma(P_{|L _1} )\cap\mathbb T$.
Indeed, if $\lambda \in \mathbb T$ is not in $\sigma(P_{L_1})$, we can take in the previous 
proof $p_0=1$.

3. While the peripheral spectra are all the same for $1<p< \infty$, the full spectra
may be different. An example (of a convolution on the unit circle) in which
$\sigma(P_{|L_p}) \ne \sigma(P_{|L_2})$ when $1<p<2$ was constructed in \cite{Sato}.

\medskip

 Bradley (private communication) constructed a symmetric ergodic Markov operator $P$
 such that the reversible stationary chain it generates is not $\alpha$-mixing, though
the peripheral spectrum of $P$ in $L_2$ is $\{1\}$ (so $P$ is Ritt). Since the chain is 
not $\alpha$-mixing, $\lim_{n\to\infty}\|P^n-E\|_2>0$, and by Proposition \ref{quasi-ritt} 
has no global CLT.

\begin{lem} \label{eigen}
Let $P$ with invariant probability $m$ be  ergodic. If $Pf=\lambda f \in L_1(m)$ for
some $\lambda \in \mathbb T$, then $f$ is bounded and $|f|$ is constant a.e.
\end{lem}
\begin{proof}We have $|f|=|Pf| \le P|f|$. Then $\int|f|dm \le \int P|f|dm \le \int |f|dm$
yields that $P|f|=|f|$ a.e. Similarly, for every $n\in \mathbb N$ we have 
$P\big(\min(|f|,n)\big)=\min(|f|,n)$, so by ergodicity $\min(|f|,n)$ is constant, which
yields that $|f|$ is constant a.e.
\end{proof}

\begin{theo} \label{totally-ergodic}
 Let $P$ with invariant probability $m$ be  ergodic and $1<p<\infty$.
Then $\|P^n-E\|_p\to 0$ (i.e. $P$ has a spectral gap in $L_p(m)$)
if and only if  $P$ is uniformly ergodic in $L_p$ and all the powers $P^k$ are ergodic.
\end{theo}
\begin{proof} Assume $\|P^n-E\|_p\to 0$. Then obviously $P$ us uniformly ergodic in $L_p$,
and $\|(P^k)^n- E\|_p \to 0$ as $n\to\infty$ shows that $P^k$ is ergodic.

Assume now that $P$ is uniformly ergodic in $L_p$ and all its powers are ergodic. Let 
$\lambda \in \mathbb T$ be an eigenvalue of $P$ on the complex $L_p(m)$. By
uniform ergodicity, $I-P$ is invertible on $L_{p,0}=(I-P)L_p$, so 1 is isolated in $\sigma(P)$.
By Gl\"uck \cite{Gl} all powers of $\lambda$ are also eigenvalues, so $\lambda$ is a root 
of unity,  $\lambda^k=1$. Let $Pf=\lambda f \in L_p(m)$. Then $P^kf =f$. By Lemma 
\ref{eigen} $f$ is bounded, and by ergodicity of $P^k$, $f$ is constant a.e., so $\lambda=1$.
This shows that $P$ has no unimodular eigenvalues $\ne 1$, so $\|P^n-E\|_p \to 0$ by 
\cite[Proposition 3.3]{DL4}.
\end{proof}
\medskip

\subsection{Finite peripheral spectrum} 
Let $T$ be a power-bounded operator on a complex Banach space.

\begin{lem}  \label{ritt-power}
Let $T$ be Ritt. Then $T^d$ is Ritt for every integer $d>1$.
\end{lem}
\begin{proof}
	Put $K:=\sup_n n\|T^n(I-T)\| <\infty$.
Fix $d$ and let $\lambda_j:=\e^{2\pi ij/d}$, $0\le j \le d-1$ be the $d$th roots of unity.
Then
$$
n\|T^{nd}(I-T^d)\|=n\|T^{nd} \prod_{j=0}^{d-1}(\lambda_j I-T)\| \le
$$
$$
d^{-1}dn\|T^{nd}(I-T)\|\cdot\|\prod_{j=1}^{d-1}(\lambda_j I-T)\| \le
d^{-1}K\,\|\prod_{j=1}^{d-1}(\lambda_j I-T)\|.
$$
Similarly, if $T$ is quasi-Ritt, so is $T^d$.
\end{proof}

If $T$ is a symmetric contraction on $H$ with $\sigma(T)=[-1,1]$, then $T^2$ is Ritt and $T$
is not.

\begin{prop}
The following are equivalent for $T$ power-bounded with $\sigma(T)\cap\mathbb T=\{1\}$.

	(i) $T$ is Ritt.

	(ii) For every integer $d\ge 1$, $T^d$ is Ritt.

	(iii) For some $d>1$, $T^d$ is Ritt.
\end{prop}
\begin{proof}
	(i) implies (ii) By Lemma \ref{ritt-power}, and obviously (ii) implies (iii).

Assume (iii). Let $\lambda_j:=\e^{2\pi ij/d}$, $0\le j\le d-1$ be the $d$th roots
of unity. Since $\sigma(T)\cap\mathbb T=\{1\}$, for $j>0$ the resolvent $R(\lambda_j,T)$
is well definded, and then
$$
n\|T^{nd}(I-T)\|= n\|T^{nd}(I-T)\prod_{j=1}^{d-1}(\lambda_j I-T)R(\lambda_j,T)\| \le
$$
$$
n\|T^{nd}(I-T^d)\|\prod_{j=1}^{d-1}\|R(\lambda_j,T)\| \le
	K_d \prod_{j=1}^{d-1}\|R(\lambda_j,T)\|:= C.
$$
	Putting $n=\ell d+r$, $0\le r <d$, we obtain  \newline
	$n\|T^n(I-T)\|\le \max_{0\le r<d}\|T^r\|(\ell+1)d\|T^{\ell d}(I-T)\|
	\le C \,\max_{0\le r<d}\|T^r\|.$
\end{proof}

\begin{theo}
Let $T$ be  power-bonuded on a complex Banach space with $r(T)=1$.

(i) If $\sigma(T)\cap\mathbb T$ is a finite set $\{\lambda_1,\dots,\lambda_k\}$, then
\begin{equation}\ \label{ritt-E}
\lim_{n\to\infty} \big\|T^n\prod_{j=1}^k(\lambda_j I-T)\big\| =0.
\end{equation}

	(ii) If for some finite set $S:=\{\lambda_1,\dots,\lambda_k\} \subset \mathbb T$
	\eqref{ritt-E} holds, then $\sigma(T)\cap\mathbb T \subset S$.
\end{theo}
\begin{proof} (i) follows from \cite[Theorem 5]{KT}.

Proof of (ii): let $\lambda \in \sigma(T)\cap\mathbb T$. By the spectral mapping theorem,
$\lambda^n\prod_{j=1}^k(\lambda_j-\lambda)$ is in
$\sigma\big(T^n\prod_{j=1}^k(\lambda_j I-T)\big)$, so \eqref{ritt-E} implies
$$
|\prod_{j=1}^k(\lambda_j -\lambda)|=|\lambda^n\prod_{j=1}^k(\lambda_j-\lambda)|\le
\|T^n\prod_{j=1}^k(\lambda_j I-T)|| \to 0.
$$
Hence $\lambda$ equals one of the $\lambda_j$.
\end{proof}

Bouabdillah and Le Merdy \cite{BLM} studied resolvent conditions which yield the convergence
\eqref{ritt-E} with rate $O(1/n)$.

\bigskip

{\bf Acknowledgements.} Part of this work was done during visits of the second author to
the university of Brest; its hospitality is gratefully acknowledged.

The authors thank Jochen Gl\"uck and Rick Bradley for their private communications,
and Guy Cohen for discussions on the Appendix. 

The first author was supported by the ANR project “Rawabranch”
number ANR-23-CE40-0008.

\end{document}